\documentclass[12pt,twoside]{article}

    \usepackage{graphicx}
    \usepackage{styleset}
    \usepackage{macros}
    \let\subsubsection\subparagraph

    \title  {The Dehn twist on a sum of two $\KK$ surfaces}

    \author {P. B. Kronheimer and T. S. Mrowka%
      \thanks{%
        The work of the first author was supported by the National
        Science Foundation through NSF grant
        DMS-1707924. The work of the second author was supported by
        NSF grant DMS-1808794.}}

    \address {Harvard University, Cambridge MA 02138 \\
              Massachusetts Institute of Technology, Cambridge MA 02139}

\begin{document}

\maketitle            

\begin{abstract}
Ruberman gave the first examples of 
self-diffeomorphisms of four-manifolds that are isotopic to the identity in
the topological category but not smoothly so.            
We give another example of this phenomenon, using the Dehn twist along a
3-sphere in the connected sum of two $\KK$ surfaces.             
\end{abstract}

\section{Introduction}

A 2-dimensional Dehn twist is a non-trivial self-diffeomorphism of an
annulus, fixing the two boundary circles pointwise.  As a
generalization of this, for any $n\ge 2$, there is a
self-diffeomorphism of the $n$-manifold $[0,1] \times S^{n-1}$,
\[
       \delta_{n} : [0,1] \times S^{n-1} \to [0,1] \times S^{n-1}
\]
having the form $ \delta_{n}(t, s) = (t, \alpha_{t}(s)), $ where
$\alpha : [0,1]\to \SO(n)$ is a loop based at the identity element
lying in the non-trivial homotopy class. We can arrange that
$\delta_{n}$ is the identity near both boundary components, and this
allows to extend $\delta_{n}$ to a diffeomorphism of any $n$-manifold
$X$ provided only that an embedding of $[0,1]\times S^{n-1}$ in $X$ is
given. The resulting diffeomorphism $\delta: X\to X$ is referred to as
a Dehn twist along the sphere $S^{n-1}\subset [0,1]\times S^{n-1}$. In
particular if $X$ is a connected sum $X_{1}\csum X_{2}$, then one can
consider the Dehn twist along the separating sphere in the neck. In
this paper, we shall prove:

\begin{theorem}\label{thm:Dehn}
    Let $Z$ be the connected sum $\KK \csum \KK$, and let $\delta : Z\to Z$
    be a Dehn twist along the separating $S^{3}$ in the neck. Then
    the diffeomorphism $\delta$ is not isotopic to the identity.
\end{theorem}

It is evident that $\delta$ induces the identity map on homology, and
by a theorem of Quinn \cite{Quinn}, applicable to closed
simply-connected four-manifolds in general, it follows that $\delta$
lies in the identity component of the homeomorphism group
$\Top(Z)$. Theorem~\ref{thm:Dehn} implies that it does not lie in the
identity component of $\Diff(Z)$. The first example of this phenomenon
-- an element in the kernel of the map
$\pi_{0}(\Diff(X))\to \pi_{0}(\Top(X))$ for a smooth 4-manifold $X$ --
was given by Ruberman in \cite{Ruberman-1}.  Additional examples were
presented later by Baraglia and Konno in \cite{Baraglia-Konno-1}. The
theorem above provides the first example where the Dehn twist on a
connected sum of simply connected 4-manifolds has
been shown to be non-trivial.

The techniques we employ here are drawn from the same toolkit that has
been used to detect other non-smoothability results for families in
dimension four: gauge theory (here the Seiberg-Witten equations), the
numerical invariants of families that they can be used to define
\cite{Ruberman-1,Ruberman-2,Kato-Konno-Nakamura}, and the homotopy
refinements of these that one may construct in the spirit of
\cite{Bauer-Furuta-1} and \cite{MingXu}.

The next statement, which is a corollary of the theorem above, was
proved earlier by Baraglia and Konno in \cite{Baraglia-Konno-2}. (The
version here is not explicitly stated in \cite{Baraglia-Konno-2}, but
an equivalent reformulation, Proposition~\ref{cor:w2} below, is a
special case of their results, as we explain in the next section.)

\begin{proposition}[\cf~\cite{Baraglia-Konno-2}]
    \label{cor:DehnBoundary}
    Let $X'$ be the 4-manifold with boundary obtained by removing an
    open ball from a $\KK$ surface. Let $\delta: X' \to X'$ be the
    diffeomorphism supported in the interior of $X'$ obtained by a
    Dehn twist along a 3-sphere parallel to the boundary. Then
    $\delta$ is not isotopic to the identity element in the group
    $\Diff(X', \partial X')$ of diffeomorphisms which fix the boundary
    pointwise.
\end{proposition}

In this paper, we will reprove Proposition~\ref{cor:DehnBoundary},
using Bauer-Furuta invariants for families in a slightly different
way, and we will deduce Theorem~\ref{thm:Dehn} using a product theorem
for connect sums modeled on \cite{Bauer-2}.

\subparagraph{Acknowledgements.} Conversations with Sander Kupers and
Danny Ruberman were very helpful for the formulation of these results
and in the preparation of this paper. In particular, the authors would like
to thank Sander Kupers for pointing out the relevance of \cite{Kreck}
in the higher-dimensional case. 

\section{Families of spin manifolds}

To explain the connection between Proposition~\ref{cor:DehnBoundary}
and the results of \cite{Baraglia-Konno-2}, consider first the action
of $\Diff(\KK)$ on the frame bundle $F(\KK)$ and the resulting map
$e: \Diff(\KK) \to F(\KK)$ which one obtains by applying this to a
basepoint $\theta$ in $F(\KK)$. The map $e$ is a fibration whose fiber
(the stabilizer of $\theta$) has the weak homotopy type of the group
$\Diff(X', \partial X')$ and the exact sequence of the fibration gives
\[
    \pi_{1}(\Diff(\KK)) \stackrel{e_{*}}{\longrightarrow}
    \pi_{1}(F(\KK)) \longrightarrow \pi_{0}(\Diff(X', \partial X')
\]    
The fundamental group of $F(\KK)$ is $\Z/2$ (as it is for any
simply-connected spin manifold), and the class of the Dehn twist
$\delta$ in $\pi_{0}(\Diff(X', \partial X'))$ is the image of the
generator of $\pi_{1}(F(\KK))$. The assertion in
Corollary~\ref{cor:DehnBoundary}, that $\delta$ is non-trivial, is
therefore equivalent to saying that the map
\begin{equation}\label{eq:e}
          e_{*} : \pi_{1}(\Diff(\KK)) \to \Z/2
\end{equation}
is zero. 

The map $e_{*}$ can be interpreted in yet another way. Given a loop
$\gamma$ in $\Diff(\KK)$, form the fiber bundle $E\to S^{2}$ with
fiber $\KK$, using $\gamma$ as the clutching function. Then
$e_{*}[\gamma]$ is equal to the evaluation of $w_{2}(TE)$ on any
section of $E$.  An elementary restatement of the vanishing of $e_{*}$
is therefore the following:

\begin{proposition}[\cite{Baraglia-Konno-2}]\label{cor:w2}
    Let $E\to S^{2}$ be a smooth fiber bundle, with fiber a $\KK$
    surface. Then the Stiefel-Whitney class $w_{2}(TE)$ is zero.
\end{proposition}

The above proposition is a consequence of \cite[Corollary
1.3]{Baraglia-Konno-2} and the arguments of \cite[Section
4.2]{Baraglia-Konno-2}. As shown in \cite{Baraglia-Konno-3}, the
corresponding statement in the topological category is false: there is
a non-smoothable topological family of $\KK$ surfaces over $S^{2}$
which has non-trivial Stiefel-Whitney class.

There is a third reformulation of Propositions~\ref{cor:DehnBoundary}
and \ref{cor:w2}, which is the one most convenient for our discussion
of the Bauer-Furuta invariants below. To set up the general context,
let $X$ be a smooth, oriented, closed 4-manifold, which we assume at
present is connected and simply connected with even intersection form.
In the absence of a Riemannian metric, a ``spin structure'' on $X$ can
be defined as a lift of the structure group of the oriented frame
bundle of $X$, from $\SL(4,\R)$ to the double cover $\tilde\SL(4,\R)$.
Our hypotheses imply that $X$ admits a spin structure $\s$, which is
unique up to isomorphism. The group of automorphisms of $\s$ is the
group of order $2$ generated by the deck transformation of the double
cover.

We write $\SDiff(X)$ for the group of \emph{orientation-preserving}
diffeomorphisms of $X$. The group $\SDiff(X)$ has a double cover
$\DiffSpin(X)\to\SDiff(X)$ consisting of pairs $(f, i)$, where
$f:X\to X$ is a diffeomorphism and $i: f^{*}(\s)\to \s$ is an
isomorphism of spin structures. Given an element $h\in \DiffSpin(X)$,
we can form a mapping torus $X^{h}$ fibering over the circle
$B=[0,1]/\sim$, together with a spin structure $\s^{h}$ on the
vertical tangent bundle of the fibration.

As a special case we can consider the element $\tau= (1, t)$ in
$\DiffSpin(X)$, where $1$ denotes the identity on $X$ and $t$ denotes
the deck transformation of the spin structure. We shall prove,

\begin{proposition}[\cite{Baraglia-Konno-2}]\label{prop:tau}
    If $X$ is a $\KK$ surface, then the involution $\tau$ in
    $\DiffSpin(X)$ is not in the identity component. Equivalently, the
    family of spin manifolds $(X^{\tau}, \s^{\tau})$ over the circle
    $B$ is not isomorphic to the trivial family $B\times (X,\s)$.
\end{proposition}

It follows from the discussion that this proposition is an equivalent
reformulation of both Propositions~\ref{cor:DehnBoundary} and
\ref{cor:w2}. Indeed, the double cover $\DiffSpin(X)\to\SDiff(X)$ is
classified by a map $\pi_{1}(\SDiff(X)) \to \Z/2$ which is easily
identified with the map $e_{*}$ in
\eqref{eq:e}. Proposition~\ref{prop:tau} says that this double cover
is trivial for $\KK$, which is equivalent therefore to the vanishing
of $e$. We will prove Proposition~\ref{prop:tau} in section
\ref{sec:twist-calc}, after introducing the tools from Seiberg-Witten
theory in section~\ref{sec:BF}.

\section{Bauer-Furuta invariants for spin families}
\label{sec:BF}

To fix our conventions and context, we summarize in this section the
techniques of finite-dimensional approximation, as applied to the
Seiberg-Witten equations on 4-manifolds, first for a single 4-manifold
as in \cite{Furuta, Bauer-Furuta-1}, and then for families of
4-manifolds over a base, as developed and explored first in
\cite{MingXu}, and later in \cite{Szymik,Baraglia-2,
  Baraglia-Konno-3}, for example.  We focus on the case that $X$ is
equipped with a spin structure, rather than a more general $\spinc$
structure, and we assume that $b_{1}(X)=0$.

So let $X$ be a closed, oriented 4-manifold with $b_{1}=0$. Let $\s$
be a spin structure on $X$. After equipping with the manifold with a
Riemannian metric, the spin structure gives rise to spin bundles
$S^{+}, S^{-}$ over $X$, Clifford multiplication
$\gamma: \Lambda^{1}\otimes S^{+}\to S^{-}$, and the Dirac operator
$D:\Gamma(S^{+})\to \Gamma(S^{-})$. The Seiberg-Witten map is a
non-linear Fredholm map between Hilbert spaces,
\[         
    \SW : \cW^{+} \to \cW^{-}.
\]
Unwrapping this a bit, we have
\[
    \cW^{+} = \cV^{+} \oplus \cU^{+} , \qquad \cW^{-} = \cV^{-} \oplus
    \cU^{-},
\]
where $\cV^{\pm}$ are suitable Sobolev completions of
$\Gamma(S^{\pm})$, and $\cU^{+}$, $\cU^{-}$ are Sobolev completions of
respectively $\Omega^{1}(X)$ and
$\Omega^{+}(X) \oplus \Omega^{0}(X)/\R$. The Seiberg-Witten map has
the form $\SW=l + c$, where $l$ is the Fredholm operator
\[
    l = D \oplus (d^{+}, d^{*})
\]
and $c$ has the form
\[
    c(a, \phi) = \bigl(\gamma(ia, \phi),\; q(\phi, \bar\phi)\bigr)
\]
where $q$ is a bilinear term.

Now let $W^{-}\subset \cW^{-}$ be a finite-dimensional subspace, large
enough that $W^{-} + \im(l) =\cW^{-}$. Set $W^{+} =
l^{-1}(W^{-})$. The corresponding \emph{finite-dimensional
approximation} to the Seiberg-Witten map is constructed in
\cite{Bauer-Furuta-1} as the map
\[
                \sw  = (\rho \comp \SW) : W^{+} \to W^{-},
\]
where
$\rho : \cW^{-} \setminus S \left ((W^{-})^{\bot}\right) \to W^{-}$ is
a suitable retraction. It is shown in \cite{Bauer-Furuta-1} that if
$W^{-}$ is sufficiently large, then the image of $\SW$ does not
intersect the unit sphere $S((W^{-})^{\bot})$ in the orthogonal
complement, so the composite $\rho\comp\SW$ is indeed defined. The
finite-dimensional approximation is a proper map and extends to the
one-point compactifications as a map of spheres:
\begin{equation}\label{eq:sw-spheres}
       [\sw] \in [ W^{+}_{\infty}, W^{-}_{\infty}]. 
\end{equation}

We will take $W^{-}$ always of the form $V^{-}\oplus U^{-}$, in which
case also $W^{+} = V^{+}\oplus U^{+}$. Furthermore, $\cV^{\pm}$ are
quaternion vector spaces and $D$ is quaternion-linear, and we are
therefore able to insist also that $V^{\pm}$ are quaternion vector
subspaces.

The operator $(d^{+} + d^{*})$ is injective with cokernel $H^{+}(X)$,
the space of harmonic self-dual 2-forms. We may choose $U^{-}$ to
contain this space, so that
\[
\begin{aligned}
    U^{-} &= u^{-} \oplus H^{+} \\
\end{aligned}
\]
and $l : U^{+} \to u^{-}$ is a linear isomorphism. A choice of
orientation of $H^{+}$ then allows us to identify the orientation
lines of $U^{+}$ and $U^{-}$. The vector spaces $V^{\pm}$ are
naturally oriented themselves, because they are quaternion vector
spaces. In all then, the orientation of $H^{+}$ allows us to identify
the orientation lines of $W^{\pm}$.

Let $M$ be a regular fiber of the finite-dimensional approximation
$\sw$, over a point in $p \in W^{-}$. As the fiber of
map between relatively oriented vector spaces, $M$ is naturally a
stably framed manifold. The results of \cite{Bauer-Furuta-1} imply
that, provided $W^{-}$ is sufficiently large, the fiber $M$ is
compact, and its framed cobordism class depends only on $X$ and the
orientation of $H^{+}$, not on the choice of Riemannian metric or on
the choice of $W^{-}$. Here ``sufficiently large'' means only that
$W^{-}$ should contain a subspace $W^{-}_{*}(g)$ which depends on the
metric $g$. The dimension of $M$ is
\[
           d = -\frac{1}{4}(2 e + 3\sigma)  + 1,
\]
where $e$ and $\sigma$ are the Euler number and signature of $X$.

Framed cobordism classes of $d$-manifolds are classified by the stable
$d$-stem $\pi^{s}_{d}= \pi_{d+N}(S^{N})$ for $N$ large. Instead of
referring to the regular fiber of $\sw$, we can refer directly to the
homotopy class of the map $\sw$ itself, as a map between spheres as in
\eqref{eq:sw-spheres}. For the purposes of this paper, we use the
following stripped-down version of the Bauer-Furuta invariant:

\begin{definition}
    The Bauer-Furuta invariant of the spin manifold $(X,\s)$, with the chosen
    orientation of $H^{+}$, is the framed cobordism class $\eta(X)$ of
    the regular fiber of $\sw$, for any metric $g$ and any choice of $W^{-}$
    containing $W^{-}_{*}(g)$.
\end{definition}

Now suppose instead of a single 4-manifold we have a smooth fiber
bundle $X \to B$ over a compact base. Let $\s$ be a given fiberwise
spin structure, and let a family of Riemannian metrics be given. The
spaces $\cW^{+}$ and $\cW^{-}$ are now bundles over $B$, and $\SW$ is
a bundle map. After choosing a suitable finite-rank subbundle $W^{-}\subset
\cW^{-}$,  we have finite-dimensional approximations $\sw :
W^{+}\to W^{-}$, where $W^{+} = V^{+} \oplus U^{+}$ and $W^{-} = V^{-}
\oplus U^{-}$ are finite-rank vector bundles over $B$. This
construction is defined whenever $W^{-}$ is sufficiently large, which
can be taken to mean that $W^{-}$ contains a certain subbundle
$W^{-}_{*}$ depending on the metric. The map $\sw$ is proper and
therefore extends to the fiberwise one-point compactifications: it
becomes a map of based sphere bundles, with a homotopy class
\[
         [\sw] \in [ W^{+}_{\infty}, W^{-}_{\infty}].
\]

Let $s$ be a smooth section of $W^{-}\to B$, transverse to $\sw$. The
inverse image $\sw^{-1}(s)$ is then a compact manifold $M$ with a map
to $B$. The dimension of $M$ is $\dim(B) + d$, where $d$ is the
invariant of the $4$-dimensional fiber, as above. 

To go further, we impose extra conditions to ensure that we can stably
trivialize the bundles $W^{\pm}$ canonically. First, the bundles
$V^{\pm}$ are quaternionic, and $\mathrm{Sp}(N)$ is 2-connected. So if
we require that $\dim B \le 2$, then these bundles have preferred
trivializations. To trivialize $U^{+} \ominus U^{-}$ stably is again
equivalent to trivializing $H^{+}$ viewed now as a bundle over
$B$. There is an action of $\pi_{1}(B)$ on the homology $H^{2}(X_{b})$
of the fiber $X_{b}$. If we impose the condition that this action is
trivial, then $H^{+}$ is a maximal positive-definite subbundle of the
trivial vector bundle $H^{2}(X_{b};\R)$, and it therefore has a
preferred trivialization as $B \times H^{+}(X_{b})$ because the space
of maximal positive-definite subspaces of $H^{2}(X_{b};\R)$ is a
contractible subset of the Grassmannian. An orientation of $H^{+}$ for
any fiber therefore completely determines a trivialization.

At this point, the map $p: M \to B$ has a relative stable framing: a
stable trivialization of $TM \ominus TB$. It is convenient in the
exposition to make $M$ itself stably framed, and to do so we ask that
$B$ have a stable framing of its tangent bundle. We are specifically
interested in the case that $B$ is the circle equipped with the stable
framing which \emph{bounds the framed disk}. Recall from the
introduction, that if $(X,\s)$ is a spin 4-manifold and $h\in \DiffSpin(X,\s)$,
then we can construct a spin family $(X^{j}, \s^{h})$
over the base $B=S^{1}$ as the
mapping torus of $h$. The Bauer-Furuta construction 
now produces a framed manifold $M(X^{h},\s^{h})$. We summarize this
as follows.

\begin{definition}\label{def:BF}
    Let $(X,\s)$ be a closed, oriented spin 4-manifold with $b_{1}=0$ and let
    $h\in\DiffSpin(X,\s)$ be an element which acts trivially on
    $H_{2}(X)$. Equip the circle $B$ with the bounding stable framing
    and let $M$ be the resulting framed manifold of dimension $1+d$,
    defined as $M  = \sw^{-1}(s)$, for a generic section $s$ of
    $W^{-}$. The Bauer-Furuta invariant $\eta(X^{h},\s^{h})$ of $h$ is the framed cobordism
    class of $M$, or equivalently the corresponding element of $\pi^{s}_{d+1}$.
\end{definition}

Our choice to give $B$ the framing which bounds means that the
invariant $\eta$ is zero for the trivial product family over $B$. So
if $\eta(X^{h}, \s^{h})$ is non-zero, then $h$ is not in the identity
component of $\DiffSpin(X,\s)$.

\begin{remark}
    As the discussion makes clear, the construction in this form
    applies equally well if $B$ is, for example, a 2-sphere. Thus, if
    $(X,\s)$  is a spin manifold with $b_{1}=0$, then the construction
    provides a homomorphism $\pi_{1}(\DiffSpin(X,\s)) \to
    \pi^{s}_{d+2}$.  
\end{remark}    

\section{Calculation for the twisted $\KK$ family}
\label{sec:twist-calc}

Let $X$ again be a closed, oriented 4-manifold with $b_{1}=0$,
equipped with a spin structure $\s$. Let $\tau\in \DiffSpin(X,\s)$ be
the deck transformation of $\s$, the generator of the kernel of
$\DiffSpin(X,\s)\to \SDiff(X)$. We can form the spin family $(X^{\tau},
\s^{\tau})$ over $B$, and the invariant $\eta(X^{\tau},
\s^{\tau})$. 

\begin{proposition}\label{prop:twist-calc}
    If the signature of $X$ is equal to $16$ mod $32$
     (i.e.~if the complex index of $D$ is $2$ mod $4$), then the
     Bauer-Furuta invariant of the twisted spin family $(X^{\tau},
     \s^{\tau})$ over the circle $B$ is given by a product
    \[
        \eta(X^{\tau},\s^{\tau}) = \eta_{1} \mathbin {\scriptstyle
          \times} \eta(X,\s),
     \]
     as a  cobordism class of stably framed $(d+1)$-manifolds. Here
     $\eta_{1}$ is the non-trivial element of
    $\pi^{s}_{1}$, represented by the circle with Lie-group framing, and
    $\eta(X,\s)$ is the Bauer-Furuta invariant of $(X,\s)$, as a
    stably framed $d$-manifold. If the signature of $X$ is equal to
    $0$ mod $32$, then $\eta(X^{\tau}, \s^{\tau})$ is zero.         
\end{proposition}

\begin{corollary}
    If $X$ is a $\KK$ surface, then $\eta(X^{\tau},\s^{\tau})$ is non-zero.
\end{corollary}

\begin{proof}[Proof of the Corollary]
    The Bauer-Furuta invariant of $\KK$ with its unique spin structure
    $\s$ is the class $\eta_{1}$, represented by the Lie-framed
    circle. The signature of $\KK$ is 16, so the proposition above
    tells us that the invariant of the family $(X^{\tau}, \s^{\tau})$
    is $\eta_{1}\times\eta_{1}$. This is the
    generator of the stable $2$-stem, $\pi^{s}_{2}=\Z/2$.
\end{proof}

Proposition~\ref{prop:tau} follows directly from this result, as do
the reformulations, Propositions~\ref{cor:DehnBoundary} and \ref{cor:w2}.

\begin{proof}[Proof of Proposition~\ref{prop:twist-calc}]
    Let \[ \sw : W^{+} \to W^{-} \] be the finite-dimensional
    approximation of the Seiberg-Witten map for the spin manifold $X$
    itself, equipped with some metric $g$. As usual we write $W^{+} =
    V^{+}\oplus U^{+}$ and similarly with $W^{-}$. The vector spaces
    $V^{\pm}$ are quaternion vector spaces, and we adopt the
    convention that the quaternion scalars $I$, $J$, $K$ act on the
    left. The circle group acts on $V^{+}$ and
    $V^{-}$ by left-multiplication by $e^{I\theta}$. We extend this
    action to all of $W^{\pm}$ by making the action trivial on
    $U^{\pm}$. The
    Seiberg-Witten map commutes with this circle action, as does its
    finite-dimensional approximation.

    Over the interval $[0,1]$, form the trivial product bundles
    $[0,1]\times W^{\pm}$. For any $\theta\in [0,\pi]$,
     let $W^{+}_{\theta} \to B$ be the vector bundle obtained by
     identifying $\{1\}\times W^{+}$ with $\{0\}\times W^{+}$ using
     left-multiplication by $e^{I\theta}$:
     \[
         \{1\}\times W^{+}\; \stackrel{e^{I\theta} \times
           \,\mathord{\cdot}}{\longrightarrow}\;\{0\}\times W^{+}.
     \]
     Define $W^{-}_{\theta}\to B$ similarly.

     The map $\sw$ commutes with $e^{I\theta}$, so it gives rise to a
     bundle map over $B$, for each $\theta$,
     \[
              \sw_{\theta} : W^{+}_{\theta} \to W^{-}_{\theta}.
          \]
      When $\theta=0$, this is (the finite-dimensional approximation to)
      the   Seiberg-Witten map for the trivial family $B\times (X,\s)$
      over the circle. When $\theta=\pi$ this is the Seiberg-Witten
      map for the twisted family $(X^{\tau}, \s^{\tau})$, because the
      involution $\tau$ acts as $-1$ on the spin bundles $S^{\pm}$.

      We now wish to compare the two proper bundle maps
      \[
          \begin{aligned}
              \sw_{0} : W^{+}_{0} &\to W^{-}_{0} \\
              \sw_{\pi} : W^{+}_{\pi} &\to W^{-}_{\pi} .
          \end{aligned}
           \]
      On the one hand we have a proper isotopy between them, given by
      the bundle maps
      \begin{equation}\label{eq:swtheta}
          \sw_{\theta} : W^{+}_{\theta}\to
          W^{-}_{\theta},
      \end{equation}
      for $\theta\in
      [0,\pi]$. However, multipliction on the left by $e^{I\theta}$ is
      not a quaternion-linear transformation for intermediate values
      $\theta\in(0,\pi)$, so we do not have an isotopy through a
      family of quaternion vector bundles $V^{\pm}_{\theta}$. The
      trivializations of $V^{\pm}_{\pi}$ arising from their
      quaternionic structure may be different then the trivializations
      they aquire from the trivial bundles $V^{\pm}_{0}$ via this
      isotopy.

      To compare the trivializations, we construct a different isotopy
      over $[0,\pi]\times B$ between the vector bundles $V^{\pm}_{0}$
      and $V^{\pm}_{\pi}$. To do so we trivialize the fiber $V^{+}$ of
      the trivial bundle $V^{+}_{0}$ as $\HH^{n}$. We construct a vector
      bundle $\hat W^{+}_{\theta}$ over
      $B$, for $\theta\in [0,\pi]$, in just the same was as we defined
      $W^{+}_{\theta}$ before, but now using right-multiplication by
      $e^{I\theta}$ on $\HH^{n}$ instead of left-multiplication. We do
      the same with $\hat W^{-}_{\theta}$. The
      Seiberg-Witten map does not commute with this action, so does
      not define a bundle map; but we are concerned only with the
      trivializations. Because right-multiplication is quaternion
      linear, we now have vector bundles $\hat
      W^{\pm}_{\theta}$, providing an isotopy over $[0,\pi]\times B$
      between $W^{\pm}_{0}$ and $W^{\pm}_{\theta}$, and the ``$V$''
      summands of these are quaternion vector bundles. The
      trivialization of $W^{\pm}_{\theta}$ that we are required to use
      is the one that arises from the trivial bundle via this new
      isotopy.

      To compare the trivializations that arise via these two
      different isotopies, consider composing the first with the
      second. On the $\HH^{n}$ summand corresponding to $V^{+}$, we have an isotopy over
      $[0,2\pi]\times B$ from the trivial bundle over $B$ with fiber
      $\R^{4n}$ back to itself. The total bundle over $[0,2\pi]\times
      B$ is constructed from the trivial bundle over $[0,2\pi] \times
      [0,1]$ by identifying the fibers using a certain path
      $[0,2\pi]\to \SO(4n)$. This path is the concatenation of a first
      path
      from $1$ to $-1$ given by left multiplication by $e^{I\theta}$
      with a second path from $-1$ to $1$ given by right
      multiplication by $e^{I\theta}$. In the case $n=1$, such a path
      from $1$ to $1$ belongs to the non-trivial homotopy class in
      $\pi_{1}(\SO(4))$. In general it is non-trivial in
      $\pi_{1}(\SO(4n))$ if and only if $n$ is odd.

      The same arguments apply to $W^{-}$ as well, so when we consider
      a relative framing of $W^{+}_{\pi} \ominus W^{-}_{\pi}$, we see
      that the stable framing acquired from the trivial bundle
      by the isotopy \eqref{eq:swtheta} is equal to the quaternionic
      framing if
      \[
            \dim_{\HH} V^{+} - \dim_{\HH} V^{-}
      \]    
      is even. Otherwise, the framings differ by the non-trivial map
      $B\to \SO$. (This difference is half the complex index of $D$.)

      In terms of the stably framed manifold $M_{\tau}$ representing the
      Bauer-Furuta invariant of the twisted family $(X^{\tau},
      \s^{\tau})$, the conclusion is that $M_{\tau}$ is
      framed-cobordant to the product $B\times M$ (and therefore to
      $\emptyset$) if the index of $D$ is $0$ mod $4$, and is
      framed-cobordant to $L\times M$ if the index is $2$ mod $4$,
      where $L$ is the circle with non-zero framing.
\end{proof}

\section{Connected sums}
\label{sec:sums}

Let $(X, \s_{X})$ and $(Y, \s_{Y})$ be two spin 4-manifolds.  We
suppose that both have $b_{1}=0$ so that our exposition of the
Bauer-Furuta invariants applies. Remove
standard balls from each and identify collars of the boundary $3$-spheres by a
diffeomorphism $\psi$ so as to form the oriented connected sum $X
\csum_{0} Y$. We write $C_{X}$ and $C_{Y}$ for these (closed) collars
and
\[
          \psi : C_{X } \to C_{Y}
      \]
for the diffeomorphism.      
Let $\alpha: [0,1]\to \SO(4)$ be a closed loop in the
non-trivial homotopy class. For each $t\in [0,1]$ we can form a
connected sum $X \csum_{t} Y$, using $\psi\comp \alpha(t)$ in place of
$\psi$. The 4-manifolds $X\csum_{0} Y$ and $X \csum_{1} Y$ are
identical, so we may form a family of 4-manifolds over the circle
$B=[0,1]/~$, which we write as
\begin{equation}\label{eq:sum-XY}
         p:   X \csum_{\alpha} Y \to B.
\end{equation}    

The derivative of $\psi$ identifies the frame bundles of the two
collars. Lift this to an isomorphism $\tilde\psi$ of the spin bundles
$\s_{X}|_{C_{X}} \to \s_{Y}|_{C_{Y}}$. There are two such lifts:
$\tilde\psi$ and $\tilde\psi\comp \tau$, where $\tau$ is the deck
transformation: choose one of them. Thought of as a path in $\Diff(C_{X})$, there is a
unique lifting of $\alpha$ to a path $\tilde\alpha$ in $\DiffSpin(C_{X}, \s_{X})$
starting at the identity. Because the homotopy class of $\alpha$ in
$\SO(4)$ is non-trivial, its lift is not a closed loop and we have
$\tilde\alpha(1)=\tau$. So if we equip $B \times X$ with the twisted
family of
spin structures $\s^{\tau}$ and equip $B\times Y$ with the product
family, then we can form a spin family
\begin{equation}\label{eq:sum-XY-spin}
            \left (X^{\tau} \csum_{\alpha} Y,\; \s_{X}^{\tau}
            \csum_{\tilde\alpha} \s_{Y} \right).
\end{equation}    
We could also do the twist $\tau$ on the side of $Y$ rather than
$X$.

\begin{proposition}
     For the family of 4-manifolds \eqref{eq:sum-XY} over the circle $B$
    with fiber $X\csum_{0}Y$, equipped with the family of spin
    structures \eqref{eq:sum-XY-spin}, the Bauer-Furuta invariant
    $\eta$ is given by
    \[
        \eta_{1} \mathbin{\scriptstyle \times} \eta(X,\s_{X})
        \mathbin{\scriptstyle \times} \eta(Y,\s_{Y})
    \]
    if $\sigma(X)$ is $16$ mod $32$, and zero otherwise.        
\end{proposition}

\begin{remark}
    There is an asymmetry in the conclusion here (only the signature
    of $X$ matters, not the signature of $Y$)
    because of the asymmetry in the construction of
    the family of spin structures.
\end{remark}

\begin{proof}
    Let us abbreviate our notation a little by taking the spin
    structures as implied. We write $X^{\tau}$ for the family over $B$
    with twisted spin structure, and we write $X^{\tau} \csum_{\alpha}
    Y$ for the result of summing $Y$ to each fiber of $X^{\tau}$ using
    the path $\alpha$ to vary the gluing. The assertion of the
    proposition can then be rephrased (using
    Proposition~\ref{prop:twist-calc}) as:
    \[
               \eta(X^{\tau} \csum_{\alpha} Y) = \eta(X^{\tau})
               \mathbin{\scriptstyle \times} \eta(Y).
   \]
    This is the product formula for the Bauer-Furuta invariants from
    \cite{Bauer-2}, extended to the case of families.

    The proof from
    \cite{Bauer-2} extends without much difficulty. To clarify this,
    we repeat the main setup steps from \cite{Bauer-2}, in the families
    context. Let $Z^{a}\to B$ be a finite collection of closed
    4-manifolds over a base $B$, indexed by $a\in A$, a finite
    set. Let $Z$ be the disjoint union. For each $a$, suppose there is
    a decomposition of $Z^{a}$ as a fiberwise connected sum, realized
    explicitly by a smooth embedding of a family of collars $C^{a}$
    over $B$. That is, $C^{a}\subset Z^{a}$ is a bundle over $B$ with fiber
    $[-1,1]\times S^{3}$, which we take to be equipped with a
    family of metrics isometric to the standard one on each fiber. Let
    $C\hookrightarrow Z$ be the union. We can write
    \[
              Z = Z_{+} \cup Z_{-},\qquad C = Z_{+} \cap Z_{-}
          \]
    where $Z_{+}$ and $Z_{-}$ are families of manifolds with boundary, with
    boundary components indexed by $A$.  Let $\sigma : C \to C$ be an
    automorphism of $C$ over the base $B$, which permutes the
    components by an even permutation of $A$. Let $Z^{\sigma}$ be
    obtained as the union of $Z_{+}$ and $Z_{-}$, attached along $C$
    using the automorphism $\sigma$.

    If $Z$ is given a fiberwise spin structure, and if $\sigma : C \to
    C$ is lifted to an isomorphism of spin families, then $Z^{\sigma}$
    also acquires the structure of a spin family over $B$. We have
    Bauer-Furuta invariants arising from finite-dimensional
    approximations $\sw$ and $\sw^{\sigma}$ for these two
    families. (For a disjoint union, the Seiberg-Witten map $\SW$ is
    defined to be the fiber product over $B$.)

    In the case
    that $B$ is a point, the construction of $Z^{\sigma}$ from $Z$ is
    the same setup as in \cite{Bauer-2}, and in Section~3 of that
    paper as series of homotopies is constructued, to show that
    that finite-dimensional approximations
    $\sw$ and $\sw^{\sigma}$ are homotopic. The same homotopies
    can be applied fiberwise over $B$, because all the estimates can
    be made uniformly over the compact base. This establishes that $\sw$ and
    $\sw^{\sigma}$ are properly homotopic bundle maps over $B$. As in
    \cite{Bauer-2}, the application to a connected sum of two manifolds is
    deduced by considering the case that $|A|=3$ and taking the fibers
    of $Z$ to be
    \[
        (X \csum S^{4}) \sqcup (S^{4} \csum Y) \sqcup (S^{4}
                 \csum S^{4}).
             \]
     A cyclic permutation $\sigma$ of order $3$ is even and the
     resulting family $Z^{\sigma}$ has fibers
     \[
                (X \csum Y) \sqcup (S^{4} \csum S^{4}) \sqcup (S^{4}
                 \csum S^{4}).
             \]
      A family of $4$-spheres over the circle has two possible spin
      structures, related by the twist $\tau$, but the resulting two
      spin families are
      isomorphic. The Bauer-Furuta invariant of either family is
      represented by the identity map between zero-dimensional vector
      bundles over $B$, which contributes trivially. So the homotopy
      between $\sw$ and $\sw^{\sigma}$ identifies the invariant of the
      family with fibers 
      $X\csum Y$ with that of the family with fibers $X \sqcup Y$.
\end{proof}

We now return to the first theorem stated in this paper.

\begin{proof}[Proof of Theorem~\ref{thm:Dehn}]
Consider the case that $X$ and $Y$ are both $\KK$ surfaces, and
form the family of 4-manifolds
\[
            p : X \csum_{\alpha} X \to B, \qquad (X = \KK).
\]
There is a unique spin structure up to isomorphism on the fiber, and
there are two ways to equip the family over $B$ with fiberwise spin
structures. In the notation of the constructions above, these are the
spin families $X^{\tau} \csum_{\alpha} X$ and
$X \csum_{\alpha} X^{\tau}$. The previous proposition says that the
Bauer-Furuta invariant for each of these families is $(\eta_{1})^{3}$,
because $\eta(X) = \eta_{1}$. The cube of $\eta_{1}$ is the element of
order $2$ in $\pi^{s}_{3}=\Z/24$, and in particular is non-zero. It
follows that the underlying family of smooth 4-manifolds
$X \csum_{\alpha} X$ is a non-trivial family over the circle. The
monodromy of this family is the mapping class of the Dehn twist
$\delta$ supported in the collar where the connected sum consruction
is made. It follows that $\delta$ is not isotopic to the identity.
\end{proof}

\section{Additional remarks}
\label{sec:additional}

\paragraph{More Dehn twists.}
The non-triviality of $\delta$ in Proposition~\ref{cor:DehnBoundary}
is a question raised in a more general form by
Giansiracusa in \cite{Giansiracusa}.
To describe this,  let $X$ be a simply-connected, closed, spin
4-manifold, and let $X^{(n)}$ be obtained from $X$ by removing $n$
disjoint balls. Let $\Diff(X^{(n)}, \partial)$ denote the group of
diffeomorphisms which are the identity in a neighborhood of the
boundary, so that there is map
\[
               \Diff(X^{(n)}, \partial) \to \Diff(X^{(n)}).
\]
It is shown in \cite{Giansiracusa} that the corresponding map on
$\pi_{0}$ has kernel equal to either $(\Z/2)^{n-1}$ or
$(\Z/2)^{n}$. The ambiguity results from the question of whether the
particular diffeomorphism $\delta^{(n)}$, defined as the composite of the Dehn
twists around the $n$ disjoint spheres parallel to the boundary
components, is isotopic to the identity in $\Diff(X^{(n)},
\partial)$. For given $X$, the answer to this question is independent
of $n$. Proposition~\ref{cor:DehnBoundary} is the statement that
$\delta^{(n)}$ is non-trivial in the case of a $\KK$ surface. As
pointed out in \cite{Giansiracusa}, the isotopy class  of
$\delta^{(n)}$ is trivial if $w_{2}(X)$ is non-zero,
and is also trivial in the case of
$S^{2}\times S^{2}$ as one can see by using a circle action on the
manifold. 

\paragraph{Homotopy {\boldmath $\KK$} surfaces.}
It follows from the results of \cite{Morgan-Szabo} and
\cite{Bauer-Furuta-1} that $\eta(X) = \eta_{1}$ for any homotopy $\KK$
surface $X$. Theorem~\ref{thm:Dehn} therefore applies equally well to
 homotopy $\KK$ surfaces.

\paragraph{Exploiting equivariance of {\boldmath $\sw$}.}

The Seiberg-Witten maps $\sw$  are
equivariant for the action of the group
$\mathit{Pin}(2) \subset \mathit{Sp}(1)$ generated by the circle $e^{I\theta}$
and the element $J$. These act by multiplication on the quaternion
vector spaces $V^{\pm}$. On $U^{\pm}$ the circle action is trivial and
$J$ acts by $-1$. This extra structure is exploited in \cite{Furuta}
and in \cite{Bauer-Furuta-1, Bauer-2}, but we have not used it here
except in an auxiliary role, to construct an isotopy in the proof of
Proposition~\ref{prop:twist-calc}. It would be interesting to see if the
equivariant version can be used to extend Theorem~\ref{thm:Dehn} to
some other cases.

\paragraph{Higher dimensions.}

In higher dimensions, the possible non-triviality of the boundary Dehn
twist in Proposition~\ref{cor:DehnBoundary} is a question addressed in
detail by Kreck in \cite{Kreck} for almost-parallelizable
$(k-1)$-connected $2k$-manifolds; and in the same setting, the results
there allow one to determine when a Dehn twist on the neck of a
connected sum is non-trivial. Specifically, let $X$ be a
$(k-1)$-connected and almost parallelizable $2k$-manifold, and let
$X'$ denote the manifold-with-boundary obtained by removing a ball.
Then in the notation of \cite{Kreck} there is assigned to $X$ an
element $\varSigma_{X}$ of order at most two in the group
$\Theta_{2k+1}$ of homotopy spheres of dimension $1$ higher. This
assignment is additive for connected sums, and if the dimension of $X$
is at least $6$, then $\varSigma_{X}$ is zero if and only if the Dehn
twist on the sphere parallel to the boundary of $X'$ is zero in
$\pi_{0}(\Diff(X', \partial X'))$.  For a connected sum
$X_{1}\# X_{2}$, it can also be deduced from \cite{Kreck} that the
Dehn twist around the neck is trivial if and only if
$\varSigma_{X_{1}}$ is zero in the quotient
$\Theta_{2k+1}/(\varSigma_{X_{1}\#X_{2}})$. Since these elements have
order at most 2, this criterion for non-triviality is simply that
$\varSigma_{X_{1}}$ and $\varSigma_{X_{2}}$ are both non-zero.
 
As an example of the computations in \cite{Kreck}, if the dimension is
$8$, then $\varSigma_{X}$ is non-zero if the index of the Dirac operator
on $X$ is odd.  If the index of the Dirac operator on $X$ is even,
then the vanishing of $\varSigma_{X}$ is dependent on the smooth
structure.  As a special case, if $X$ is the exotic sphere of
dimension $8$, then the Dehn twist in the neck of $X\#X$ is a
non-trivial element of $\pi_{0}(\Diff(X\#X))$. In this case, the
connected sum $X\#X$ is $S^{8}$ and it is presented as the union of
two balls: the non-triviality of the Dehn twist arises from the
non-standard parametrization of the standard $7$-sphere along which
the Dehn twist is performed.

\bibliographystyle{plain}
\bibliography{k3-families}

\end{document}